\documentclass[10pt]{amsart}
\usepackage{amsmath}
\usepackage{amssymb}
\usepackage{amsfonts}
\usepackage{amsthm}
\usepackage{enumerate}
\usepackage[all]{xy}
\usepackage[latin1]{inputenc}
\usepackage{mathdots}
\usepackage{graphicx}
\usepackage{latexsym}
\usepackage{mathrsfs}

\usepackage{tikz}

\PassOptionsToPackage{usenames,dvipsnames,svgnames}{xcolor}

\input xy

\newtheorem{Theo}{Theorem}[section]

\newtheorem{Prop}[Theo]{Proposition}
\newtheorem{Cor}[Theo]{Corollary}
\newtheorem{Lemma}[Theo]{Lemma}

\theoremstyle{definition}

\newtheorem{Exam}[Theo]{Example}

\newtheorem{Defn}[Theo]{Definition}

\def\mystrut(#1,#2){\vrule height #1pt depth #2pt width 0pt}
\newcommand{\rep}{{\rm rep}}
\newcommand{\rrep}{\overline{\mystrut(5,0) {\rm rep\hspace{-1pt}}}\,}

\newcommand{\Hom}{{\rm Hom}}

\newcommand{\End}{{\rm End}}
\newcommand{\mmod}{{\rm mod}}
\newcommand{\Mod}{{\rm Mod}}

\newcommand{\Z}{\mathbb{Z}}

\newcommand{\A}{\mathcal{A}}
\newcommand{\B}{\mathcal{B}}

\begin{document}

\author{Charles Paquette}
\address{Charles Paquette, Department of Mathematics, University of Connecticut, Storrs, CT, 06269-3009, USA. Phone: 1-860-486-1286}
\email{charles.paquette@uconn.edu}



\begin{abstract} Let $\A$ be a Hom-finite additive Krull-Schmidt $k$-category where $k$ is an algebraically closed field. Let $\mmod \A$ denote the category of locally finite dimensional $\A$-modules, that is, the category of covariant functors $\A \to \mmod k$. We prove that an irreducible monomorphism in $\mmod \A$ has a finitely generated cokernel, and that an irreducible epimorphism in $\mmod \A$ has a finitely co-generated kernel. Using this, we get that an almost split sequence in $\mmod \A$ has to start with a finitely co-presented module and end with a finitely presented one. Finally, we apply our results to the study of $\rep(Q)$, the category of locally finite dimensional representations of a strongly locally finite quiver. We describe all possible shapes of the Auslander-Reiten quiver of $\rep(Q)$.\end{abstract}

\title[Locally finite dimensional representations]{Irreducible morphisms and locally finite dimensional representations}

\maketitle

\section{Introduction}

Throughout this note, $k$ stands for an algebraically closed field. All categories are preadditive $k$-categories. Let $\A$ be a Hom-finite additive Krull-Schmidt $k$-category, that is, a Hom-finite additive $k$-category in which the Krull-Remak-Schmidt decomposition theorem holds for every object. Using Gabriel-Roiter's terminology from \cite{GR}, $\A$ is sometimes called an \emph{aggregate}. By a \emph{module} over $\A$, we mean an additive covariant functor $\A \to \Mod k$, where $\Mod k$ stands for the category of $k$-vector spaces. The category of all such functors with the usual morphisms of functors (that is, the natural transformations) is denoted by ${\rm Mod} \A$.  It is well known (and easy to check) that the category ${\rm Mod} \A$ is an abelian $k$-category. Note that if $\A'$ is a skeleton of $\A$, then the categories ${\rm Mod} \A$ and ${\rm Mod} \A'$ are equivalent. Furthermore, since our functors are additive between additive categories, it is sufficient to define the functors on the indecomposable objects of $\A'$. Let $\mathfrak{S}(\A)$ be a full subcategory of $\A$ whose objects form a complete set of representatives of the isomorphism classes of indecomposable objects of $\A$. Following Gabriel-Roiter \cite{GR}, the category $\mathfrak{S}(\A)$ is called a \emph{spectroid} of $\A$. Since the categories $\Mod \A$ and $\Mod \mathfrak{S}(\A)$ are equivalent, we will rather work with $\Mod \mathfrak{S}(\A)$ and assume from the beginning that the category $\A$ is a \emph{spectroid}, that is, a Hom-finite (but not additive) $k$-category in which all objects have local endomorphism algebras, and distinct objects are not isomorphic. We finally denote by $\mmod \A$ the full subcategory of $\Mod \A$ of all additive covariant functors $\A \to \mmod k$, where $\mmod k$ stands for the category of finite dimensional $k$-vector spaces. A module in $\mmod \A$ is called \emph{locally finite dimensional}. A typical example of a spectroid is the path category $kQ$ of an interval finite quiver $Q$, that is, a quiver $Q=(Q_0, Q_1)$ such that for any pair of vertices $x,y \in Q_0$, there are finitely many paths from $x$ to $y$. Indeed, it follows from the definition of a path category that $kQ$ is Hom-finite if and only if $Q$ is interval finite. In this case, the category $\Mod \A$ is the category ${\rm Rep}(Q)$ of $k$-linear representations of $Q$ and the category $\mmod \A$ is the category ${\rm rep}(Q)$ of locally finite dimensional $k$-linear representations of $Q$.

\medskip

The motivation behind this paper was to get general statements for non-existence of almost split sequences that would generalize the results obtained in \cite{Pa} in the setting of categories of the form $\mmod kQ/I$ where $kQ/I$ is a spectroid. It turns out that only by using new properties of irreducible morphisms in $\mmod \A$, we can restrict the starting term and ending term of an almost split sequence in $\mmod \A$. Indeed, if $f: M \to N$ is an irreducible monomorphism in $\mmod \A$, then the cokernel ${\rm coker}(f)$ of $f$ is finitely generated. The stronger property that ${\rm coker}(f)$ is finitely presented is not true in general. However, if $M$ is finitely generated or if $N \to {\rm coker}(f)$ is irreducible, then ${\rm coker}(f)$ is finitely presented. We get a similar statement for an irreducible epimorphism $M \to N$ with $N \in \mmod \A$. As a consequence, we get the following.

{\renewcommand{\theTheo}{\ref{Theorem1}}
\begin{Theo} Let $0 \to L \to M \to N \to 0$ be an almost split sequence in $\mmod \A$. Then $L$ is finitely co-presented and $N$ is finitely presented.
\end{Theo}
\addtocounter{Theo}{-1}}

Conversely, it has been proven by Auslander \cite{Aus_Survey} that if $N$ is indecomposable non-projective and finitely presented, then there is an almost split sequence ending in $N$ in $\mmod \A$. Similarly, if $L$ is indecomposable non-injective and finitely co-presented, then there is an almost split sequence starting in $L$ in $\mmod \A$. Hence, the above theorem tells us that Auslander gave all the almost split sequences in $\mmod \A$.

\medskip

If we specialize to the hereditary case, we get an interval finite quiver $Q$ with a faithful and surjective functor $kQ \to \A$, which is an isomorphism if and only if the infinite radical of $\A$ vanishes. Assuming that the infinite radical vanishes and that the simple modules are finitely presented and finitely co-presented, we get $\A \cong kQ$ where $Q$ is strongly locally finite. In this case, $\mmod \A$ is the category ${\rm rep}(Q)$ of locally finite dimensional $k$-linear representations of $Q$. Let $M$ be indecomposable in $\rep(Q)$. If $M$ is not finitely presented, we show that there is at most one indecomposable object $L$ with an irreducible morphism $L \to M$ and there is no irreducible morphism $L^2 \to M$. If $M$ is not finitely co-presented, we show that there is at most one indecomposable object $N$ with an irreducible morphism $M \to N$ and there is no irreducible morphism $M \to N^2$. Using this together with the results obtained in \cite{Pa_rrep} allow us to describe all possible shapes of the Auslander-Reiten quiver of $\rep(Q)$. By a \emph{quasi-wing}, we mean a full and convex subquiver of $\Z\mathbb{A}_\infty$ generated by quasi-simple vertices and by a \emph{linear quiver}, we mean a connected subquiver of $\cdots \to \circ \to \circ \to \cdots$. A \emph{star quiver} is a finite quiver to which we attach a finite number of rays and co-rays with no common vertices. We get the following, which is a simplified version of Theorem \ref{Theorem2}.

{\renewcommand{\theTheo}{(Simplified version of \ref{Theorem2})}
\begin{Theo} Let $Q$ be connected infinite and strongly locally finite. The Auslander-Reiten quiver of $\rep(Q)$ has the following connected components.
\begin{enumerate}[$(1)$]
    \item A unique preprojective component $\mathcal{P}_Q$, which is a full predecessor closed subquiver of $\mathbb{N}Q^{\rm op}$.
\item A unique preinjective component $\mathcal{I}_Q$, which is a full successor closed subquiver of $\mathbb{N}^-Q^{\rm op}$.
\item An infinite number of quasi-wings, if $Q$ is not infinite Dynkin. Otherwise, a finite number of quasi-wings.
\item Some additional linear quivers, if and only if $Q$ is not a star quiver.
\end{enumerate}
\end{Theo}
\addtocounter{Theo}{-1}}

Note that some full, abelian, Hom-finite and Krull-Schmidt subcategories of $\rep(Q)$ have been studied in \cite{BLP, Pa_rrep}, with a description of their Auslander-Reiten quivers. In \cite{BLP}, the Auslander-Reiten quiver of the category $\rep^+(Q) \subseteq \rep(Q)$ of the finitely presented representations of $Q$ is described. By duality, this also gives a description of the Auslander-Reiten quiver of the category $\rep^-(Q)\subseteq \rep(Q)$ of the finitely co-presented representations. In \cite{Pa_rrep}, a full subcategory $\rrep(Q)$ of $\rep(Q)$ is introduced, which contains both $\rep^+(Q), \rep^-(Q)$. This category is abelian, Hom-finite and Krull-Schmidt, and the description of the Auslander-Reiten quiver of $\rrep(Q)$ is given as in the above theorem, where only points $(1), (2), (3)$ are considered. It is also shown in \cite{Pa_rrep} that the connected components of the Auslander-Reiten quiver of $\rrep(Q)$ are connected components of the Auslander-Reiten quiver of $\rep(Q)$. Therefore, the above theorem closes the loop in the description of the Auslander-Reiten quiver of $\rep(Q)$ by adding point $(4)$. It is worth noting that the category $\rep(Q)$ is, in general, neither Hom-finite nor Krull-Schmidt.

\section{Basic facts, quivers and hereditary categories}

In this first introductory section, we collect some basic facts about $\mmod \A$ and we later specialize to the case where $\mmod \A$ is hereditary which happens, for instance, when $\A = kQ$ with $Q$ interval finite.

\medskip

Let $M,N \in \mmod \A$ and let $f: M \to N$ be a morphism. Recall that $f$ is explicitly given by a family $$(f_a: M(a) \to N(a))_{a \in \A_0}$$ of linear maps such that for any $\alpha \in \A(c,d)$, we have $f_dM(\alpha) = N(\alpha)f_c$. There is a duality $\mathcal{D}: \mmod \A \to \mmod \A^{\rm op}$, which is defined as follows. Let $D = \Hom_k(-, k)$ be the standard duality.  For $a \in \A_0$, we set $\mathcal{D}(M)(a) = D(M(a))$ and for $\alpha \in \A(c,d)$, we set $\mathcal{D}(M)(\alpha) = D(M(\alpha)): D(M(d)) \to D(M(c))$. This defines a module $\mathcal{D}(M)$ in $\mmod \A^{\rm op}$. Now, we define a morphism $\mathcal{D}(f): \mathcal{D}(N) \to \mathcal{D}(M)$ as follows. For $a \in \A_0$, we set $\mathcal{D}(f)_a = D(f_a): \mathcal{D}(N)(a) \to \mathcal{D}(M)(a)$. Using the functoriality of the standard duality $D$, it is easy to see that $(\mathcal{D}(f)_a)_{a \in \A_0}$ defines a morphism $\mathcal{D}(f)$ from $\mathcal{D}(N)$ to $\mathcal{D}(M)$. This makes $\mathcal{D}$ a functor $\mmod \A \to \mmod \A^{\rm op}$ which is a duality.

\medskip

Observe that for $a \in \A_0$, the projective object $P_a:=\Hom_\A(a,-)$ lies in $\mmod \A$ and is indecomposable. It is called the \emph{projective module at} $a$. Dually, $I_a:=\mathcal{D}\Hom_\A(-,a)$ lies in $\mmod \A$, is indecomposable and is injective in $\mmod \A$. It will be called the \emph{injective module at} $a$. An indecomposable object $M$ in $\mmod \A$ is \emph{finitely generated} if there exist finitely many objects $a_1, \ldots, a_r$ in $\A$ such that we have an epimorphism $\oplus_{i=1}^rP_{a_i} \to M$ and \emph{finitely presented} if, moreover, the kernel of this epimorphism is also finitely generated. Dually, an indecomposable object $M$ in $\mmod \A$ is \emph{finitely co-generated} if there exist finitely many objects $b_1, \ldots, b_s$ in $\A$ such that we have a monomorphism $M \to \oplus_{i=1}^sI_{b_i}$ and \emph{finitely co-presented} if, moreover, the cokernel of this monomorphism is also finitely co-generated. It is clear that $M$ is finitely generated (resp. finitely presented) in $\mmod \A$ if and only if $\mathcal{D}(M)$ is finitely co-generated (resp. finitely co-presented) in $\mmod \A^{\rm op}$.

\medskip

The following is an easy observation that will be handy in the sequel. It can be found in \cite{GR}.

\begin{Prop}[Gabriel-Roiter]
Every indecomposable object in $\mmod \A$ has a local endomorphism algebra.
\end{Prop}

\begin{proof} We give an outline of the proof for the sake of completeness. Let $M \in \mmod \A$ be indecomposable and let $f: M \to M$ be a morphism. For each $a \in \A_0$, the map $f_a : M(a) \to M(a)$ admits a diagonal decomposition $f_a = f_{a,i} \oplus f_{a,n}: M(a,i) \oplus M(a,n) \to M(a,i) \oplus M(a,n)$ where $f_{a_i}$ is an isomorphism and $f_{a,n}$ is nilpotent. It is not hard to check that the decompositions $M(a)=M(a,i) \oplus M(a,n)$, for $a \in \A_0$, are such that for any $c,d \in \A_0$ and $\alpha \in \A(c,d)$, $M(\alpha): M(c,i) \oplus M(c,n) \to M(d,i) \oplus M(d,n)$ is a diagonal map. This yields a decomposition $M = M_i \oplus M_n$. If $M_n = 0$, then $f_a$ is an isomorphism for all $a \in \A_0$, which makes $f$ an isomorphism. If $M_i=0$, then $f$ is such that $f_a$ is nilpotent for all $a \in \A_0$. In such a case, observe that $1_M + f + f^2 + \cdots$ defines a morphism $M \to M$ which is the inverse of $1_M-f$. Therefore, we see that $\End(M)$ is local.
\end{proof}

The next lemma is easy but very useful.

\begin{Lemma} \label{LemmaEasy}
Let $M \in \Mod \A$ be finitely generated. Then $M \in \mmod \A$ and $\Hom(M,N)$ is finite dimensional whenever $N \in \mmod \A$.
\end{Lemma}

\begin{proof} There are objects $a_1, \ldots, a_r \in \A_0$ such that there is an epimorphism $$\textstyle{\bigoplus}_{1 \le i \le r}P_{a_i} \to M.$$ For $b \in \A_0$, we have a surjective map $\textstyle{\bigoplus}_{1 \le i \le r}P_{a_i}(b) \to M(b)$, hence a surjective map $\textstyle{\bigoplus}_{1 \le i \le r}\Hom_\A(a_i,b) \to M(b)$. Since $\A$ is Hom-finite, $M(b)$ is finite dimensional. Let $N \in \mmod \A$. We have an injective map
$$\Hom(M,N) \to \textstyle{\bigoplus}_{1 \le i \le r}\Hom(P_{a_i},N).$$
Now, by Yoneda's lemma, for $1 \le i \le r$, the space $\Hom(P_{a_i},N) \cong N(a_i)$ is finite dimensional. Therefore, $\Hom(M,N)$ is finite dimensional.
\end{proof}

Recall that the \emph{radical} ${\rm rad}\,\B$ of a category $\B$ is an ideal of $\B$ such that $\alpha \in \B(a,b)$ lies in ${\rm rad}\,\B$ if and only if, for any $\beta \in \B(b,a)$, $1_a-\beta\alpha$ is an isomorphism. The \emph{infinite radical} ${\rm rad}^\infty \B$ of $\B$ is the intersection $\cap_{i \ge 1}{\rm rad}^i\B$ of all powers of ${\rm rad}\B$. A morphism $\alpha \in \B(a,b)$ is \emph{irreducible} if it is neither a section nor a retraction and whenever $\alpha$ factors as $\alpha = \gamma\beta$, then $\beta$ is a section or $\gamma$ is a retraction. If $a,b$ have local endomorphism algebras, then $\alpha \in \B(a,b)$ is irreducible if and only $\alpha \in {\rm rad}\,\B$ but $\alpha \not\in {\rm rad}^2\B$. Hence, in a spectroid with a vanishing infinite radical, every morphism is a sum of composition of irreducible morphisms.

Sometimes, we will consider the case where $\A$ is \emph{hereditary}, that is, when $\Mod \A$ is a hereditary category.

\begin{Exam} Consider the quiver $Q$ where $Q_0 = \{x_1, x_2, \ldots\} \cup \{y_1, y_2,\cdots\}$, and we have the following arrows: there is a single arrow $x_i \to x_j$ if and only if $j=i+1$ and this arrow is denoted $\alpha_i$; and there is a single arrow $y_s \to y_t$ if and only if $s = t+1$ and this arrow is denoted $\beta_s$. Moreover, for any $i,j$, there is a single arrow $\gamma_{j,i}:x_i \to y_j$. The quiver $Q$ is depicted below.
$$
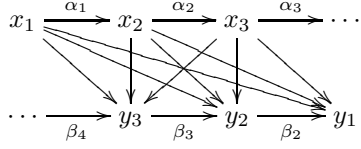
\xymatrix{x_1 \ar[r]^{\alpha_1} \ar[dr] \ar[drr] \ar[drrr] & x_2 \ar[r]^{\alpha_2} \ar[d] \ar[dr] \ar[drr] & x_3 \ar[r]^{\alpha_3} \ar[dl] \ar[d] \ar[dr] & \cdots \\ \cdots \ar[r]_{\beta_4} & y_3 \ar[r]_{\beta_3} & y_2 \ar[r]_{\beta_2} & y_1}$$
We consider $I$ the ideal of $kQ$ generated by all possible $\gamma_{j,i}-  \gamma_{j,i+1}\alpha_i$ and all possible $\gamma_{j-1,i}- \beta_j\gamma_{j,i}$. We define $\A$ to be the quotient $kQ/I$. The irreducible morphisms in $\A$ are given by the arrows $\alpha_i$ and $\beta_j$. Consider the quiver $\Gamma$ obtained from $Q$ by removing all arrows $\gamma_{i,j}$. We have a faithful dense functor $k\Gamma \to \A$ and the infinite radical ${\rm rad}^\infty \A$ of $\A$ is the ideal of $\A$ generated by all arrows $\gamma_{i,j}$.  Hence, we get an equivalence $k \Gamma \to \A/{\rm rad}^\infty \A$. Clearly, $k\Gamma$ is hereditary. One can also check that $\A$ is hereditary. Observe that the projective module $P_{x_1}:= \Hom_\A(x_1,-)$ has a submodule $M$ which is indecomposable and not finitely generated.
\end{Exam}

In order to avoid this \emph{bad} behavior on the submodules of the finitely generated projective modules, most of our hereditary categories will have a vanishing infinite radical. In this case, we get the following, which is essentially due to Gabriel-Roiter; see \cite{GR}.

\begin{Prop}[Gabriel-Roiter]
Suppose that $\A$ is hereditary with a vanishing infinite radical. Then there exists an interval finite quiver $Q$ such that $\A$ is isomorphic to the path category $kQ$.
\end{Prop}

\begin{proof} The fact that there exists a quiver $Q$ such that $\A$ is isomorphic to the path category $kQ$ follows from a proposition in \cite[Chapter 8]{GR}, as $k = \bar k$ and ${\rm rad}^\infty \A=0$. If $Q$ is not interval finite, then there are two vertices $x,y \in Q_0$ with infinitely many paths from $x$ to $y$, which makes $kQ \cong \A$ not Hom-finite, which is contrary to our assumption on $\A$.
\end{proof}

For a hereditary category $\A = kQ$ with $Q$ interval finite, we will often restrict to the case where $Q$ is locally finite, for the following reason. A quiver which is locally finite and interval finite is called \emph{strongly locally finite}.

\begin{Prop}
Suppose that $\A$ is hereditary with a vanishing infinite radical, hence $\A \cong kQ$ for an interval finite quiver $Q$. Then the simple modules in $\Mod \A$ are finitely presented and finitely co-presented if and only if $Q$ is locally finite.
\end{Prop}

\begin{proof}
Since $Q$ has no oriented cycles, an object $M$ in $\Mod \A$ is simple if and only if it is non-zero on exactly one object $a$, and in this case, $M(a)$ is one dimensional. Therefore, the simple objects in $\Mod \A$, up to isomorphism, are indexed by $\A_0$. Fix $a \in \A_0$ and let $S_a$ denote the simple object in $\Mod \A$ such that $S_a(a)$ is non-zero. Let ${\rm Succ}(a)$ be the set of all arrows starting at $a$. Observe that we have a short exact sequence
$$0 \to \textstyle{\bigoplus}_{\alpha \in {\rm Succ}(a)}P_{h(\alpha)} \to P_a \to S_a \to 0$$
where $h(\alpha)$ denotes the head of $\alpha$. Thus, $S_a$ is finitely presented if and only if ${\rm Succ}(a)$ is finite. Similarly, $S_a$ is finitely co-presented if and only if there are finitely many arrows ending in $a$.
\end{proof}

Let $M \in \mmod \A$ where $\A = kQ$ is the path category of a strongly locally finite quiver $Q$. The \emph{top} of $M$, denoted ${\rm top}M$, is a quotient of $M$ defined as follows. Let $a \in \A_0$ and consider the finite set ${\rm Pred}(a)$ of all arrows ending in $a$ in $Q$. Consider the subspace $Z(a):=\sum_{\alpha \in {\rm Pred}(a)} {\rm Im}M(\alpha)$ in $M(a)$. Doing this for all vertex of $Q$ defines uniquely a submodule $Z$ of $M$, called the \emph{radical} of $M$ and denoted ${\rm rad} M$. The top of $M$ is the semi-simple quotient ${\rm top}M:= M/{\rm rad} M$. Consider the projection $M \to {\rm top}M$ and for each $a \in \A_0$, choose a subspace $T_M(a)$ of $M(a)$ such that the restriction $T_M(a) \to ({\rm top}M)(a)$ is an isomorphism. Consider the submodule $M'$ of $M$ generated by all the elements in $T_M(a)$ for $a \in \A_0$. The module $M'$ is called a \emph{top submodule} of $M$. We say that $M$ \emph{has a top} if there exists a top submodule $M'$ of $M$ with $M' = M$. Not every module has a top, as the next example shows.

\begin{Exam}
Consider the infinite quiver $$Q: \;\; \cdots \to \circ \to \circ \to \circ$$ and the module $M \in \mmod kQ$ such that $M(x)=k$ for all $x \in Q_0$ and $M(\alpha) = 1$ for all arrows $\alpha \in Q_1$. Then ${\rm top}M=0$ and any top submodule of $M$ is zero.
\end{Exam}

\begin{Lemma} \label{LemmaProjSub} Let $\A = kQ$ where $Q$ is strongly locally finite. Let $M$ be a submodule of a finitely generated projective module. Then
\begin{enumerate}[$(1)$] \item $M$ has a top and all top submodules of $M$ are equal to $M$;
\item $M$ is a direct sum of finitely generated projective modules.
\end{enumerate}
\end{Lemma}

\begin{proof}
There are vertices $x_1, \ldots, x_r$ such that $M$ is a submodule of $P:=\textstyle{\bigoplus}_{1 \le i \le r}P_{x_i}$. Let $M'$ be a top submodule of $M$ with defining spaces $T_M(a)$ for $a \in \A_0$. Fix $a \in Q_0$. Consider the full subquiver $\Sigma$ of $Q$ of all predecessors of $a$ which are successors of some $x_i$. Clearly, $\Sigma$ is convex and finite, since $Q$ is interval finite. Consider the restrictions $M_\Sigma, M'_\Sigma, P_\Sigma$ of $M, M', P$ to $\Sigma$, respectively. Now, observe that ${\rm rad}M_\Sigma = ({\rm rad}M)_\Sigma$ and hence, ${\rm top}M_\Sigma = ({\rm top}M)_\Sigma$. For $b \in \Sigma_0$, the subspace $T_M(b)$ of $M_\Sigma (b) = M(b)$, as in the definition of $M'$, is such that the restriction $T_M(b) \to ({\rm top}M_\Sigma)(b)$ is an isomorphism. Hence $M'_\Sigma$ is a top submodule of $M_\Sigma$. Since $\Sigma$ is a finite quiver without oriented cycles, this is well known that $M_\Sigma$ has a top and any top submodule of $M_\Sigma$ has to coincide with $M_\Sigma$. Therefore, $M'_\Sigma = M_\Sigma$. In particular, $M'(a) = M(a)$. Since this is true for all $a \in \A_0$, we get $M' = M$. This proves the first part. Consider now a surjective morphism $Q \to {\rm top}M$ where $Q$ is a direct sum of finitely generated projective modules such that the induced morphism ${\rm top}Q \to {\rm top}M$ is an isomorphism. It follows from the fact that $Q$ is interval finite that $Q$ is locally finite dimensional. Using a similar argument as above, $Q$ has a top and $Q'=Q$ for any choice of top submodule $Q'$ of $Q$. Using the fact that $Q, M$ are projective modules, this yields two morphisms $f: M \to Q$ and $g: Q \to M$ such that the induced morphisms ${\rm top}M \to {\rm top}Q$ and ${\rm top}Q \to {\rm top}M$ are isomorphisms. Since ${\rm Im}g$ defines a top submodule of $M$, it follows from the first part that $g$ is surjective. Similarly, $f$ is surjective. Now, $fg: Q \to Q$ is surjective. Since $Q$ is locally finite dimensional, $fg$ is an isomorphism. Therefore, $g$ is a section and $f$ is a retraction. Considering $gf: M \to M$, we get that $gf$ is an isomorphism and hence that $g$ is a retraction and $f$ is a section. Hence, $f, g$ are isomorphisms and $M$ is isomorphic to $Q$, as wanted.
\end{proof}

\section{Irreducible morphisms}

In this section, $\A$ is a spectroid. In some statements, we will put more restrictions on $\A$ to get stronger statements. This first proposition is one of the core results of this section.

\begin{Prop} \label{PropMain}
Let $f: L \to M$ be an irreducible monomorphism in $\mmod \A$. Then the cokernel ${\rm coker}(f)$ of $f$ is finitely generated.
\begin{enumerate}[$(1)$]
    \item If $L$ is finitely generated, then ${\rm coker}(f)$ is finitely presented.
\item If $M \to {\rm coker}(f)$ is also irreducible, then ${\rm coker}(f)$ is finitely presented.
    \item If $\A$ is hereditary and given by a strongly locally finite quiver with $L$ a finite direct sum of indecomposable modules, then ${\rm coker}(f)$ is finitely presented.
\end{enumerate}
\end{Prop}

\begin{proof} We may assume that $f$ is an inclusion and $N = M/L$. For proving the main part of the statement, assume to the contrary that $N$ is not finitely generated.  Let $E = \{X_i \mid i \in I\}$ be the set of all modules $X_i$ such that $L$ is a submodule of $X_i$, $X_i$ is a submodule of $M$ and $X_i/L$ is finitely generated. Now, the set $E$ with inclusions form a directed system and $\lim\limits_{\rightarrow}X_i = M$. Similarly, $\lim\limits_{\rightarrow}X_i/L = M/L = N$. Note that every $X_i$ is a proper submodule of $M$ as, otherwise, $N$ would be finitely generated. For $i \in I$, since $f$ is irreducible, and the inclusion $X_i \to M$ is not an epimorphism, the inclusion $L \to X_i$ is a section and hence, the short exact sequence
$$\eta_i: \quad 0 \to L \to X_i \to X_i/L \to 0$$
splits. This gives an exact sequence $\Hom_\A(\eta_i, L)$ for each $i \in I$. For $i \in I$, since $X_i/L$ is finitely generated and $L \in \mmod \A$, by Lemma \ref{LemmaEasy}, the space $\Hom_\A(X_i/L, L)$ is finite dimensional. By the Mittag-Leffler condition on inverse limits, we get an exact sequence
$$0 \to \lim\limits_{\longleftarrow} \Hom_\A(X_i/L, L) \to  \lim\limits_{\longleftarrow} \Hom_\A(X_i, L)  \to \lim\limits_{\longleftarrow} \Hom_\A(L, L) \to 0$$
which can be rewritten as
$$0 \to \Hom_\A(\lim\limits_{\rightarrow}X_i/L, L) \to  \Hom_\A(\lim\limits_{\rightarrow}X_i, L) \to \Hom_\A(L, L) \to 0$$
and hence as
$$0 \to \Hom_\A(N, L) \to  \Hom_\A(M, L)  \to \Hom_\A(L, L) \to 0.$$
Thus, there exists $g \in \Hom_\A(M, L)$ such that $1_L = \Hom_\A(f,L)(g) = gf$, meaning that $f$ is a section, a contradiction. This proves the main part of the proposition.

For proving (1) and (2), assume to the contrary that $N$ is not finitely presented (but is finitely generated). There exists a short exact sequence
$$0 \to R \to P \to N \to 0$$
where $P$ is projective finitely generated and $R$ is not finitely generated. For any finite set $S$ of elements of $R$, consider the quotient $R_S:=R/\langle S \rangle$. The set of all such modules with projections $R_{S'} \to R_{S}$ for $S'$ a subset of $S$ form a directed set.
We have $\lim\limits_{\longrightarrow} R_S = 0$. Set $N_S$ the cokernel of the inclusion $\langle S \rangle \to P$. We have a short exact sequence
$$0 \to R_S \to N_S \to N \to 0$$
Moreover, $\lim\limits_{\longrightarrow} N_S = N$. Now, we have a pullback diagram
$$\xymatrix{0 \ar[r] & L \ar[r]^{s_S} \ar@{=}[d] & E_S \ar[r]\ar[d]^{p_S} & N_S \ar[r] \ar[d] & 0 \\ 0 \ar[r] & L \ar[r]^f & M \ar[r]^g & N \ar[r] & 0}$$
Since $f$ is irreducible, for any $S$, either $s_S$ is a section or $p_S$ is a retraction. Assume that $p_S$ is a retraction for sets $S$ of arbitrarily large cardinality. For each such $S$, we get an exact sequence
$$(*): \quad 0 \to L \to M \oplus R_S \to N_S \to 0$$
In the situation of $(1)$, both $L, N_S$ are finitely generated, and hence $M \oplus R_S$ is finitely generated. Therefore, $R_S$ is finitely generated, which gives that $R$ is finitely generated, a contradiction. In the situation of $(2)$, since $\lim\limits_{\longrightarrow} N_S = N$, we may assume that the finite sets $S$ are large enough so that in the pushout diagram
$$\xymatrix{0 \ar[r] & L \ar[r]^f \ar[d] & M \ar[r]^g \ar[d] & N \ar[r] \ar@{=}[d] & 0 \\ 0 \ar[r] & R_S \ar[r] & N_S \ar[r] & N \ar[r] & 0}$$
corresponding to $(*)$, the morphism $M \to N_S$ is not a section. Hence, since $M \to N$ is irreducible, the morphism $N_S \to N$ is a retraction. Thus, $N_S = N \oplus R_S$ which makes $R_S$ finitely generated. As above, this gives that $R$ is finitely generated, a contradiction.
Therefore, in the situations of $(1)$ and $(2)$, we may assume that there exists a positive integer $r$ such that $s_S$ is a section for all $S$ with $|S| \ge r$. For $S$ with $|S| \ge r$, since $N_S$ is finitely generated, the space $\Hom(N_S,L)$ is finite dimensional by Lemma \ref{LemmaEasy}. Since each
$$\eta_S: \quad 0 \to L \to E_S \to N_S \to 0$$ splits, we get an exact sequence
$$\Hom(\eta_S, L): \quad 0 \to \Hom(N_S,L) \to \Hom(E_S,L) \to \Hom(L,L) \to 0$$
whenever $|S| \ge r$. Taking the inverse limit, and using the Mittag-Leffler condition on inverse limits, we get an exact sequence
$$0 \to \lim\limits_{\longleftarrow}\Hom(N_S,L) \to \lim\limits_{\longleftarrow}\Hom(E_S,L) \to \Hom(L,L) \to 0$$
which is the same as
$$0 \to \Hom(\lim\limits_{\longrightarrow} N_S,L) \to \Hom(\lim\limits_{\longrightarrow}E_S,L) \to \Hom(L,L) \to 0$$
and thus the same as
$$0 \to \Hom(N,L) \to \Hom(M,L) \to \Hom(L,L) \to 0$$
showing that $f$ is a section, a contradiction.

For proving (3), assume to the contrary that $N$ is not finitely presented, that $L$ is a finite direct sum of indecomposable modules and that $\A = k\Gamma$ for a strongly locally finite quiver $\Gamma$. Using the above notation and Lemma \ref{LemmaProjSub}, the module $R$ is an infinite direct sum of indecomposable finitely generated projective modules, and we have an infinite sequence of projections of non-finitely generated projective modules
$$R \to R_1 \to R_2 \to \cdots$$
whose direct limit is zero. We define $W_i = R/R_i$, which is projective finitely generated such that $R = R_i \oplus W_i$. We get morphisms $s_i: L \to E_i$ and $p_i: E_i \to M$. If $p_i$ is a retraction for some $i$, then we have a short exact sequence
$$0 \to L \to M\oplus R_i \to N_i \to 0$$
Since $M \oplus R_i \to N_i$ is an epimorphism and the restriction $R_i \to N_i$ is a radical map, the restriction $M \to N_i$ is an epimorphism. Hence, the co-restriction $L \to R_i$ is an epimorphism. Since $R_i$ is projective, $R_i$ is a direct summand of $L$. Now, $L$ is a finite direct sum of modules with local endomorphism algebras. By a theorem of Azumaya (see \cite{Az}), $R_i$ has to be a finite direct sum of indecomposable modules, a contradiction. Hence, we may assume that all $s_i$ are sections and we proceed as in the proof of (1) and (2).
\end{proof}

Note that given an irreducible monomorphism $L \to M$ in $\mmod \A$, it is not true, in general, that the cokernel is finitely presented, as shown in the next example.

\begin{Exam} Consider the infinite quiver $Q$ given by
$$
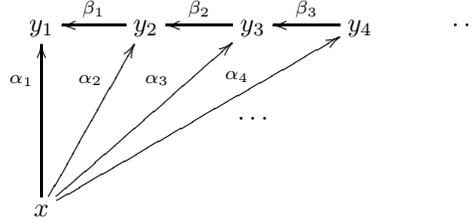
\xymatrix{y_1 & y_2 \ar[l]_{\beta_1} & y_3 \ar[l]_{\beta_2} & y_4 \ar[l]_{\beta_3} & \cdots\\ && \cdots \\ x \ar[uu]^(0.7){\alpha_1} \ar[uur]^(0.66){\alpha_2} \ar[uurr]^(0.63){\alpha_3} \ar[uurrr]^(0.65){\alpha_4} }$$
and consider the ideal $I$ of the path category $kQ$ generated by the relations $\alpha_i-\beta_i\alpha_{i+1}$. Take $\A = kQ/I$. Note that $Q$ is not the Gabriel quiver of $\A$ since for all $i \ge 1$, there is no irreducible map from $x$ to $y_i$. It is not hard to check that $\Mod \A$ is hereditary. Indeed, $M \in \Mod \A$ is projective if and only if $M(\beta_i), M(\alpha_i)$ are injective maps for all $i$. This yields that submodules of projective modules are projective, so $\Mod \A$ is hereditary. Consider the projective module $P_x:=\Hom_\A(x,-)$ which is clearly locally finite dimensional. Consider the unique maximal submodule $M$ of $P_x$ generated by all $P_x(y_i)$. Consider the cokernel $f: P_x \to S_x$ of this inclusion, where $S_x$ is the simple module at $x$. We claim that the inclusion $M \to P_x$ is irreducible. It follows from Proposition 2.7 in \cite{AR_IV} that this inclusion is irreducible if and only if, for any morphism $g:Z \to S_x$, either $g$ factors through $f$ or $f$ factors through $g$. If $g$ is nonzero, then $g$ is an epimorphism and since $P_x$ is projective, $f$ factors through $g$. If $g=0$ then clearly, $g$ factors through $f$. Hence $M \to P_x$ is irreducible. Note that the cokernel of this irreducible map is finitely generated but not finitely presented. Also, Lemma \ref{LemmaProjSub} fails in this setting.
\end{Exam}

Using the duality $\mathcal{D}: \mmod \A \to \mmod \A^{\rm op}$, we get the dual of Proposition \ref{PropMain}.
%The next proposition seems dual to Proposition \ref{PropMain}. However, \emph{a priori}, we do not know whether $M$ lies in $\mmod \A$ in the statement. %Therefore, there is no straightforward proof using only a duality argument.

\begin{Prop} \label{PropMainDual}
Let $f: M \to N$ be an irreducible epimorphism in $\mmod \A$. Then the kernel ${\rm ker}(f)$ of $f$ is finitely co-generated.
\begin{enumerate}[$(1)$]
    \item If $N$ is finitely co-generated, then ${\rm ker}(f)$ is finitely co-presented.
\item If ${\rm ker}(f) \to M$ is also irreducible, then ${\rm ker}(f)$ is finitely co-presented.
    \item If $\A$ is hereditary and given by a strongly locally finite quiver with $N$ a finite direct sum of indecomposable modules, then ${\rm ker}(f)$ is finitely co-presented.
\end{enumerate}
\end{Prop}

\section{almost split sequences}

The following nice fact follows directly from propositions \ref{PropMain} and \ref{PropMainDual}. This partially answers a conjecture of Krause in \cite{Kr} and generalizes the results in \cite{Pa}.

\begin{Theo} \label{Theorem1}Let $0 \to L \to M \to N \to 0$ be an almost split sequence in $\mmod \A$. Then $L$ is finitely co-presented and $N$ is finitely presented.
\end{Theo}

\begin{proof} Since the sequence is almost split, the two morphisms $L \to M$ and $M \to N$ are irreducible. It follows from Proposition \ref{PropMain} that $N$ is finitely presented and by Proposition \ref{PropMainDual} that $L$ is finitely co-presented.
\end{proof}

Hence, using an existence theorem of Auslander in \cite{Aus_Survey} and its dual version, we get the following consequence.

\begin{Cor} Let $M$ be indecomposable in $\mmod \A$.
\begin{enumerate}[$(1)$]
    \item If $M$ is not projective, then there is an almost split sequence ending in $M$ in $\mmod \A$ if and only if $M$ is finitely presented.
    \item If $M$ is not injective, then there is an almost split sequence starting in $M$ in $\mmod \A$ if and only if $M$ is finitely co-presented.
\end{enumerate}
\end{Cor}

\section{The category $\rep(Q)$}

Assume now that $\A$ is hereditary and given by a strongly locally finite quiver $Q = (Q_0, Q_1)$. The category $\mmod \A$ with $\A = kQ$ is the category of locally finite dimensional representations of $Q$. It will be denoted by $\rep(Q)$, which is a more convenient notation. We let $\rrep(Q)$ denote the full subcategory of $\rep(Q)$ of those representations $M$ that are the middle term of a short exact sequence
$$0 \to L \to M \to N \to 0$$
where $L$ is finitely presented, $N$ is finitely co-presented and there are finitely many arrows $\alpha$ such that $M(\alpha) \ne 0$ but $(N \oplus L)(\alpha)=0$. Surprisingly, this category is abelian (but not extension closed). As shown in \cite{Pa_rrep}, most of the Auslander-Reiten theory of $\rep(Q)$ lies in $\rrep(Q)$, in the sense that the almost split sequences in $\rep(Q)$ all lie in $\rrep(Q)$. Moreover, any irreducible morphism between indecomposable representations of $\rep(Q)$ with one term in $\rrep(Q)$ has to lie in $\rrep(Q)$.

\medskip

In general, there might be indecomposable objects in $\rep(Q)$ which are not in $\rrep(Q)$. We will show later that, for $Q$ connected, this happens if and only if $Q$ is not a star quiver.

\begin{Defn}. The quiver $Q$ is called a \emph{star quiver} provided that there exists a finite full and convex subquiver $\Gamma$ of $Q$ and two finite disjoint sets of vertices $S(\to), (\to)S$ of $\Gamma$ such that $Q$ is obtained from $\Gamma$ by adding vertices ${v_{s,i}}$ for $s \in S(\to), (\to)S$ and $i \ge 1$ with the following arrows. If $s \in S(\to)$, the arrows $\alpha_{s,0}: s \to v_{s,1}$ and $\alpha_{s,i}: v_{s, i} \to v_{s, i+1}$. If $s \in (\to)S$, the arrows $\alpha_{s,0}: v_{s,1} \to s$ and $\alpha_{s,i}: v_{s, i+1} \to v_{s, i}$.
\end{Defn}

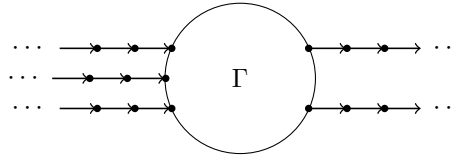
\begin{figure}[h]
  \centering
  \begin{tikzpicture}[xscale=2,yscale=2]

    \draw (0,0) circle (0.5cm);

\fill (-0.495,0) circle (0.7pt);
\fill (-0.455,0.2) circle (0.7pt);
\fill (-0.455,-0.2) circle (0.7pt);

\fill (0.455,0.2) circle (0.7pt);
\fill (0.455,-0.2) circle (0.7pt);

\node at (0,0) {$\Gamma$};

\draw [semithick,->] (0.455,-0.2) -- (0.7,-0.2);
\fill (0.71,-0.2) circle (0.7pt);
\draw [semithick,->] (0.71,-0.2) -- (0.95,-0.2);
\fill (0.96,-0.2) circle (0.7pt);
\draw [semithick,->] (0.96,-0.2) -- (1.2,-0.2);
\node at (1.4,-0.2) {$\cdots$};

\draw [semithick,->] (0.455,0.2) -- (0.7,0.2);
\fill (0.71,0.2) circle (0.7pt);
\draw [semithick,->] (0.71,0.2) -- (0.95,0.2);
\fill (0.96,0.2) circle (0.7pt);
\draw [semithick,->] (0.96,0.2) -- (1.2,0.2);
\node at (1.4,0.2) {$\cdots$};

\draw [semithick,->] (-0.7, 0.2) -- (-0.469,0.2);
\fill (-0.70,0.2) circle (0.7pt);
\draw [semithick,->] (-0.95,0.2) -- (-0.71, 0.2);
\fill (-0.95,0.2) circle (0.7pt);
\draw [semithick,->] (-1.2,0.2) -- (-0.96,0.2);
\node at (-1.4,0.2) {$\cdots$};

\draw [semithick,->] (-0.7, -0.2) -- (-0.469,-0.2);
\fill (-0.70,-0.2) circle (0.7pt);
\draw [semithick,->] (-0.95,-0.2) -- (-0.71, -0.2);
\fill (-0.95,-0.2) circle (0.7pt);
\draw [semithick,->] (-1.2,-0.2) -- (-0.96,-0.2);
\node at (-1.4,-0.2) {$\cdots$};

\draw [semithick,->] (-0.75, 0) -- (-0.505,0);
\fill (-0.75,0) circle (0.7pt);
\draw [semithick,->] (-1,0) -- (-0.76, 0);
\fill (-1,0) circle (0.7pt);
\draw [semithick,->] (-1.25,0) -- (-1.01,0);
\node at (-1.43,0) {$\cdots$};

  \end{tikzpicture}
\caption{A star quiver}
\label{fig:starquiver}
\end{figure}

In order to study the representation theory of star quivers, we need to introduce more terminology. Let $\alpha: i \to j$ be an arrow of $Q$. The \emph{tail} of $\alpha$, denoted $t(\alpha)$, is the vertex $i$ and its \emph{head}, denoted $h(\alpha)$, is the vertex $j$. We denote by $\alpha^{-1}$ a formal inverse of $\alpha$.  We define the tail $t(\alpha^{-1})$ of $\alpha^{-1}$ to be $j$ and its head $h(\alpha^{-1})$ to be $i$. Now, a \emph{walk} $w$ in $Q$ is a formal expression $w = \cdots \alpha_3 \alpha_2 \alpha_1$, which may be infinite, such that each $\alpha_i$ is an arrow or the inverse of an arrow with $h(\alpha_i) = t(\alpha_{i+1})$ for all $i$. An infinite walk is said to be \emph{simple} if all of the $t(\alpha_i)$ are distinct. If $w = \alpha_r \cdots \alpha_2 \alpha_1$ is finite, it is called \emph{simple} if all of the $t(\alpha_1), \ldots, t(\alpha_r), h(\alpha_r)$ are distinct. Note that a simple walk cannot contain both an arrow and its inverse, hence has to be reduced. Given a simple walk $w$, one attaches to $w$ a representation $M(w)$ as follows. For $x \in Q_0$, $M(x)=k$ if $x$ is a vertex of $w$ and $M(x)=0$, otherwise. For $\alpha \in Q_1$, $M(\alpha)=1$ if $\alpha$ or $\alpha^{-1}$ occurs in $w$ and $M(\alpha)=0$, otherwise. It is not hard to check that $M(w)$ is indecomposable and $\End(M(w))=k$. An \emph{infinite sourced path} in $Q$ is an infinite path with a starting vertex and an \emph{infinite sinked path} in $Q$ is an infinite path with an ending vertex. Therefore, since $Q$ has no oriented cycles, we can identify infinite sourced paths with infinite simple walks with all edges being arrows. Similarly, we can identify infinite sinked paths with infinite simple walks with all edges being inverses of arrows.
The following lemma is essential. It appears in \cite{Pa_rrep}.

\begin{Lemma} \label{Lemmarrep}Let $M \in \rrep(Q)$. Then there exists a subrepresentation $L$ of $M$ which is finitely generated projective such that $M/L$ is finitely co-presented.
\end{Lemma}

\begin{Lemma} \label{LemmaDegree}Let $p$ be an infinite sourced or sinked path in $Q$ such that there are infinitely many vertices of $p$ which are of degree at least three in $Q$. Then there is an indecomposable representation $E$ in $\rep(Q)$ which is not in $\rrep(Q)$.
\end{Lemma}

\begin{proof} We will assume that $p$ is an infinite sourced path, since the proof of the other case is dual. Let $v_1, v_2, \ldots$ be an infinite list of distinct vertices of $p$ such that each $v_i$ is of degree at least $3$ in $Q$. Consider first the case where infinitely many of these $v_i$ are the starting vertex of at least two arrows. With no loss of generality, we may assume that all of the $v_i$ are the starting vertex of at least two arrows. Consider the indecomposable representation $M(p)$, which is clearly finitely generated. Clearly, any non-zero submodule of $M(p)$ is not projective. Therefore, it follows from Lemma \ref{Lemmarrep} that $M(p)$ does not lie in $\rrep(Q)$. Consider now the case where infinitely many of these $v_i$ are the ending vertex of at least two arrows. With no loss of generality, we may assume that all of the $v_i$ are the ending vertex of at least two arrows. For each $i$, let $\alpha_i: w_i \to v_i$ be an arrow which does not occur in $p$. There is an extension
$$0 \to M(p) \to E \to \bigoplus_{i \ge 1}S_{w_i} \to 0$$
where $E$ is indecomposable and $S_{w_i}$ is the simple representation at $w_i$. Since $Q$ is locally finite, $E$ lies in $\rep(Q)$. The module $E$ is not finitely generated since $\textstyle{\bigoplus}_{i \ge 1}S_{w_i}$ is not. If $P$ is any finitely generated subrepresentation of $E$, then $E/P$ is not finitely co-generated. Hence, it follows from Lemma \ref{Lemmarrep} that $E$ does not lie in $\rrep(Q)$.
\end{proof}

\begin{Lemma} \label{LemmaWalk}Let $w$ be an infinite simple walk such that for any factorization $w=w_2w_1$, the walk $w_2$ is not an infinite path. Then $M(w) \in \rep(Q)$ but not in $\rrep(Q)$.
\end{Lemma}

\begin{proof} Assume to the contrary that $M(w)$ lies in $\rrep(Q)$. Observe that if $L$ is a finitely generated subrepresentation of $M(w)$, then $L$ is finite dimensional. Similarly, if $N$ is a finitely co-generated quotient of  $M(w)/L$, then $N$ is finite dimensional. This contradicts Lemma \ref{Lemmarrep}.
\end{proof}

Given a finite simple walk $w = \alpha_r \cdots \alpha_2 \alpha_1$, we define the \emph{tail} of $w$ to be the tail of $\alpha_1$ and the \emph{head} of $w$ to be the head of $\alpha_r$.

\begin{Prop} \label{Proprrep}
Let $Q$ be connected and strongly locally finite. Then there is an indecomposable object in $\rep(Q)$ which is not in $\rrep(Q)$ if and only if $Q$ is not a star quiver.
\end{Prop}

\begin{proof}
Assume that $Q$ is not a star quiver. Fix $x$ a vertex of $Q$. By using Lemmas \ref{LemmaDegree} and \ref{LemmaWalk}, we may assume that any infinite simple walk starting at $x$ eventually becomes an infinite path whose vertices are all of degree $2$.  Let $S$ be the set of all infinite simple walks starting at $x$. Since $Q$ is not a star quiver, the set $S$ is infinite.  For each $w \in S$, let $w_p$ be the least finite simple walk such that $w=p_ww_p$ where $p_w$ is an infinite path with all vertices occurring in $p$ of degree $2$. Let $v(w)$ denote the head of $w_p$ and let $n(w)$ be the length of $w_p$. Suppose that the set $\{n(w) \mid w \in S\}$ is finite. Let $r > 0$ such that $n(w) < r$ for all $w \in S$. By König's lemma, there is a finite simple walk $\rho$ of length $r+1$ and an infinite subset $S'$ of $S$ such that, for any $w' \in S'$, we have a factorization $w' = \rho_{w'}\rho$. Therefore, for any $w' \in S'$, the walk $\rho_{w'}$ is an infinite path whose vertices are all of degree $2$. In particular, the elements in $S'$ are all equal, a contradiction. Therefore, we may assume that the set $\{n(w) \mid w \in S\}$ is infinite. By König's lemma again, there exists an infinite simple walk $\gamma = \cdots \alpha_2 \alpha_1$ starting at $x$ and an infinite subset $\{\gamma_1, \gamma_2, \ldots\}$ of $S$ such that for $i \ge 1$, $\gamma, \gamma_i$ both start with $\alpha_i \cdots\alpha_2\alpha_1$. By assumption, there exists $s>0$ such that $\cdots \alpha_{s+1}\alpha_s$ is an infinite path with all $t(\alpha_i)$, $i \ge s$, of degree $2$. We see then that $\gamma_{j} = \gamma_s$ whenever $j \ge s$, which contradicts that $\{\gamma_1, \gamma_2, \ldots\}$ is infinite.

Assume now that $Q$ is a star quiver. Then there is a finite full and convex subquiver $\Gamma$ of $Q$ and two finite disjoint sets of vertices $S(\to), (\to)S$ of $\Gamma$ such that $Q$ is obtained from $\Gamma$ by adding vertices $v_{s,i}, i \ge 1,$ for $s \in S(\to), (\to)S$ with the following arrows. For $s \in S(\to)$, the arrows $\alpha_{s,0}: s \to v_{s,1}$ and $\alpha_{s,i}: v_{s, i} \to v_{s, i+1}$ and; for $s \in (\to)S$, the arrows $\alpha_{s,0}: v_{s,1} \to s$ and $\alpha_{s,i}: v_{s, i+1} \to v_{s, i}$. Let $M$ be indecomposable in $\rep(Q)$. It is easy to check that $M(v_{s,i})$ is injective if $s \in (\to)S$ and surjective if $s \in S(\to)$. Therefore, since $S(\to) \cup (\to)S$ is finite, there exists an integer $r > 0$ such that $M(\alpha_{s,i})$ is an isomorphism whenever $i \ge r$ and $s \in S(\to) \cup (\to)S$. Let $\Sigma_1$ be the full subquiver of $Q$ generated by the vertices $v_{s,i}$ for $i \ge r$ and $s \in S(\to)$ and $\Sigma_2$ be the full subquiver of $Q$ generated by the vertices $v_{s,i}$ for $i \ge r$ and $s \in (\to)S$. Let $P$ be the restriction of $M$ to $\Sigma_1$ and $I$ be the restriction of $M$ to $\Sigma_2$. Clearly, $P$ is finitely generated projective and $I$ is finitely co-generated injective. Let $\Sigma_3$ denote the full subquiver of the support of $M$ of the vertices not in $\Sigma_1 \cup \Sigma_2$. Let $L$ denote the restriction of $M$ to $\Sigma_3$, which is finite dimensional. We have short exact sequences
$$0 \to P \to N \to L \to 0$$
$$0 \to N \to M \to I \to 0$$
with any vertex in the support of $I \oplus P$ of degree at most $2$ in $Q$. Since $P$ is finitely presented and $L$ is finite dimensional, $N$ is finitely presented.
Now, if $\beta$ is an arrow with $M(\beta) \ne 0$, then $N(\beta) \ne 0$, $I(\beta) \ne 0$ or else $\alpha$ starts at a vertex $v_{s,r}$ for $s \in (\to)S$. Since $N$ is finitely presented and $I$ is finitely co-presented, this gives that $M \in \rrep(Q)$.
\end{proof}

\medskip

The following result is a consequence of Propositions \ref{PropMain} and \ref{PropMainDual}; compare \cite[Cor 7.2]{Pa}.

\begin{Prop} \label{PropIrredMorphisms}Let $f: L \to M$ be an irreducible morphism in $\rep(Q)$.
\begin{enumerate}[$(1)$]
    \item If $f$ is a monomorphism and $L$ is indecomposable, then ${\rm coker}(f)$ is finitely presented.
\item If $f$ is an epimorphism and $M$ is indecomposable, then ${\rm ker}(f)$ is finitely co-presented.
\end{enumerate}
\end{Prop}

The following nice fact will be very useful in the description of the Auslander-Reiten quiver of $\rep(Q)$. For $L,M \in \rep(Q)$ and $i$ a positive integer, we denote by ${\rm rad}^i(L,M)$ the morphisms $L \to M$ that lie in ${\rm rad}^i(\rep(Q))$.

\begin{Prop} \label{LastProp}
Let $M \in \rep(Q)$ be indecomposable which is neither finitely presented nor finitely co-presented. Then there is at most one indecomposable object $L$ in $\rep(Q)$ with an irreducible morphism $L \to M$. In this case, ${\rm rad}(L,M)/{\rm rad}^2(L,M)$ is one dimensional as a right $\End(L)/{\rm rad}(L,L)$-module.
\end{Prop}

\begin{proof}
Suppose that the result is not true. Then there are two indecomposable representations $L_1, L_2$ and two irreducible morphisms $f_1: L_1 \to M, f_2: L_2 \to M$ such that if $L_1 \cong L_2$, then $f_1, f_2$ are linearly independent as elements of the right $\End(L_1)/{\rm rad}(L_1,L_1)$-module ${\rm rad}(L_1,M)/{\rm rad}^2(L_1,M)$. In particular, we get an irreducible morphism $F:=(f_1, f_2): L_1 \oplus L_2 \to M$; see \cite[Prop. 3.2, Prop. 3.3]{Bau}. If both $f_1, f_2$ are epimorphism, we get exact sequences
$$0 \to K_1 \to K \to L_2 \to 0$$
$$0 \to K_2 \to K \to L_1 \to 0$$
where $K_i$ is the kernel of $f_i$ and $K$ is the kernel of $F$. Note that $K, K_1, K_2$ are finitely co-presented. The full subcategory of $\mmod \A$ of the finitely co-presented representations is abelian and extension closed; see for instance \cite[Prop. 2.1]{Aus}. Hence, each of $L_1, L_2$ is finitely co-presented which makes $M$ finitely co-presented, a contradiction. Suppose now that both $f_1, f_2$ are monomorphism. If, in addition, $F$ is also a monomorphism, we get exact sequences
$$0 \to L_2 \to M/L_1 \to C \to 0$$
$$0 \to L_1 \to M/L_2 \to C \to 0$$
where $C$ is the cokernel of $F$. Since all of $C, M/L_1, M/L_2$ are finitely presented and the full subcategory of $\rep(Q)$ of the finitely presented representations is abelian and extension closed; see for instance \cite[Prop. 2.1]{Aus}, we get that $L_1, L_2$ are finitely presented and hence $M$ finitely presented, a contradiction. If $f_1, f_2$ are monomorphism but $F$ is an epimorphism, then each $L_i$ is an extension of a finitely presented module by a finitely co-presented one. By \cite[Lemma 3.3, Theorem 6.10, Proposition 7.5]{Pa_rrep}, $L_1, L_2 \in \rrep(Q)$. Since the kernel of $F$ is finitely co-presented, it also lies in $\rrep(Q)$. Since $\rrep(Q)$ is abelian, $M \in \rrep(Q)$. If follows from \cite[Theorem 6.10]{Pa} that $L_1 = L_2$ and ${\rm rad}(L_1,M)/{\rm rad}^2(L_1,M)$ is one dimensional as a right $\End(L_1)/{\rm rad}(L_1,L_1)$-module. Now, assume that $f_1$ is a monomorphism and $f_2$ is an epimorphism. With a similar argument, we get that $L_2$ is an extension of a finitely presented module by a finitely co-presented one. By \cite[Proposition 7.5]{Pa}, $M$ is finitely presented or finitely co-presented, a contradiction.
\end{proof}

\begin{Lemma} Let $M,N \in \rep(Q)$ be indecomposable. Then the dimension of ${\rm rad}(M,N)/{\rm rad}^2(M,N)$ as a right $\End(M)/{\rm rad}(M,M)$ module is the same as the dimension of ${\rm rad}(M,N)/{\rm rad}^2(M,N)$ as a right $\End(N)/{\rm rad}(N,N)$ module.
\end{Lemma}

\begin{proof} Assume that ${\rm rad}(M,N)/{\rm rad}^2(M,N)$ is non-zero. It follows from \cite[Prop. 7.5]{Pa_rrep} that either both $M,N$ lie in $\rrep(Q)$ or both $M,N$ do not lie in $\rrep(Q)$. In the first case, the result follows from the description of the Auslander-Reiten quiver of $\rrep(Q)$ in \cite[Section 6]{Pa_rrep}. In the second case, the dimension of ${\rm rad}(M,N)/{\rm rad}^2(M,N)$ as a right $\End(M)/{\rm rad}(M,M)$ module is one by Proposition \ref{LastProp}. By the dual of Proposition \ref{LastProp}, the dimension of ${\rm rad}(M,N)/{\rm rad}^2(M,N)$ as a left $\End(N)/{\rm rad}(N,N)$ module is also one.
\end{proof}

\begin{Lemma} \label{LemmaCompIrred} Let $h$ be a composition of irreducible morphisms between indecomposable objects in $\rep(Q)$. Then the kernel of $h$ is finitely co-presented and its cokernel is finitely presented.
\end{Lemma}

\begin{proof}
We only prove the first part of the statement. Let $M_1, \ldots, M_{n+1}$ be indecomposable objects in $\rep(Q)$ with irreducible morphisms $f_i: M_{i+1} \to M_{i}$ for $1 \le i \le n$. Consider $h = f_{1} \circ f_2  \circ \cdots  \circ f_n$. We prove by induction on $n$ that $h$ has a finitely co-presented kernel. The case where $n=1$ follows from Proposition \ref{PropIrredMorphisms}. Assume that $n>1$ and let $g = f_{1} \circ \cdots \circ f_{n-1}$, which, by induction, has a finitely co-presented kernel. Set $f = f_n$, which has also a finitely co-presented kernel. Assume first that $f$ is an epimorphism. We have a short exact sequence
$$0 \to {\rm ker}f \to {\rm ker}gf  \to {\rm ker}g \to 0$$
Being an extension of two finitely co-presented representations, the representation ${\rm ker}gf$ is also finitely co-presented. Assume now that $f$ is a monomorphism. With no loss of generality, we may assume that $M_{n+1}$ is a subrepresentation of $M_n$. Observe that ${\rm ker}gf = {\rm ker}g \cap M_{n+1}$. We have a monomorphism
$$\frac{{\rm ker}g}{{\rm ker}g \cap M_{n+1}} \hookrightarrow M_n/M_{n+1} \cong {\rm coker}f.$$
Since ${\rm coker}f$ is finitely presented by Proposition \ref{PropIrredMorphisms}, the support of ${\rm coker}f$ does not contain infinite sinked-paths. Therefore, the same holds for the support of $\frac{{\rm ker}g}{{\rm ker}g \cap M_{n+1}}$. Similarly, since ${\rm ker}g$ is finitely co-presented, the support of ${\rm ker}g$ does not contain infinite sourced-paths. Thus, the same holds for the support of $\frac{{\rm ker}g}{{\rm ker}g \cap M_{n+1}}$. This shows that $\frac{{\rm ker}g}{{\rm ker}g \cap M_{n+1}}$ is finite dimensional. This together with the fact that ${\rm ker}g$ is finitely co-presented yield that ${\rm ker}g \cap M_{n+1}$ is finitely co-presented. This finishes the proof of the Lemma.
\end{proof}

\begin{Prop} \label{LastPropCycle}
Let $M_1, \ldots, M_{n+1}$ be indecomposable objects in $\rep(Q)$ with irreducible morphisms $f_i: M_{i} \to M_{i+1}$ for $1 \le i \le n$. Then $M_1$ is not isomorphic to $M_{n+1}$.
\end{Prop}

\begin{proof}
The proposition holds when all representations lie in $\rrep(Q)$; see \cite{Pa_rrep}. Moreover, if one $M_i$ lies in $\rrep(Q)$, then all $M_i$ lie in $\rrep(Q)$ by a remark at the beginning of this section. Therefore, we may assume that no $M_i$ lies in $\rrep(Q)$. Assume to the contrary that $M_1 \cong M_{n+1}$. We may then assume that $M_1 = M_{n+1}$. Consider the endomorphism $f:=f_n\circ\cdots\circ f_1$ of $M_1$. It follows from Lemma \ref{LemmaCompIrred} that $f$ has a finitely co-presented kernel and a finitely presented cokernel. Since the $f_i$ are irreducible, $f$ is not an isomorphism. Since ${\rm End}(M_1)$ is local, it means that each $f_x$ for $x \in Q_0$ is nilpotent. Therefore, for $M_1(x) \ne 0$, the map $f_x$ is neither injective nor surjective. In particular, the representations ${\rm ker}f, {\rm coker}f$ have the same support as $M_1$. Since ${\rm ker}f$ is finitely co-presented, its support has no infinite sourced-paths. Similarly, since ${\rm coker}f$ is finitely presented, its support has no infinite sinked-paths. Therefore, $M_1$ is finite dimensional and in particular, lies in $\rrep(Q)$, a contradiction.
\end{proof}

The results obtained so far allow us to describe combinatorially the irreducible morphisms in $\rep(Q)$. More precisely, we can describe the shapes of the Auslander-Reiten quiver, $\Gamma_{\rep(Q)}$, of $\rep(Q)$, which is defined as follows. The vertices of $\Gamma_{\rep(Q)}$ is given by a complete set of representatives of the indecomposable representations of $\rep(Q)$. If $M, N$ are two vertices of $\Gamma_{\rep(Q)}$, then the number of arrows $M \to N$ is the dimension of ${\rm rad}(M,N)/{\rm rad}^2(M,N)$ as a right $\End(M)/{\rm rad}(M,M)$ module, which is the same as the dimension of ${\rm rad}(M,N)/{\rm rad}^2(M,N)$ as a right $\End(N)/{\rm rad}(N,N)$ module. Since both $M,N$ have local endomorphism algebras, it is not hard to check that a non-isomorphism $f: M \to N$ is non-zero seen as an element in ${\rm rad}(M,N)/{\rm rad}^2(M,N)$ if and only if it is irreducible. Therefore, the quiver $\Gamma_{\rep(Q)}$ is a combinatorial description of the irreducible morphisms in $\rep(Q)$.

\medskip

Before stating the main theorem, we need to define the shapes of the possible connected components of $\Gamma_{\rep(Q)}$. Let $\Delta$ be the translation quiver $\mathbb{Z}\mathbb{A}_\infty$ with translation $\tau$. A \emph{quasi-simple} vertex of $\Delta$ is a vertex having only one immediate predecessor and only one immediate successor. Let $a_0$ be a fixed quasi-simple vertex in $\Delta$. The complete list of pairwise distinct quasi-simple vertices of $\Delta$ is $\{a_i:=\tau^i(a_0) \mid i \in \Z\}$. Let $I$ be an interval of $\Z$. The \emph{quasi-wing} $I$, denoted $W_I$, is the full convex subquiver of $\Delta$ generated by the $a_i$ for $i \in I$. If $I$ is bounded above but not below, then $W_I$ is a \emph{right infinite quasi-wing}. If $I'$ is another interval with the same property, then $W_{I'}$ is isomorphic to $W_I$. If $I$ is bounded below but not above, then $W_I$ is a \emph{left infinite quasi-wing}. If $I'$ is another interval with the same property, then $W_{I'}$ is isomorphic to $W_I$. If $I$ is finite, then $W_I$ is a \emph{finite quasi-wing}. Clearly, such a finite quasi-wing only depends on the length of $I$. Finally, if $I = \Z$, then $W_I = \Delta$. By a \emph{quasi-wing}, we mean a quiver of the form $W_I$ for some non-empty interval $I$ of $\Z$. Notice that if $I$ has one element, then $W_I$ is a single vertex and is also called a \emph{trivial component}.

\medskip

A full connected subquiver of $\cdots \circ \to \circ \to \circ \to \cdots$ is called a \emph{linear quiver}. These quivers play a special role in the description of $\Gamma_{\rep(Q)}$ when $Q$ is not a star quiver. As a consequence of Propositions \ref{LastProp}, \ref{LastPropCycle} and \ref{Proprrep} and \cite[Section 6]{Pa_rrep}, we get a description of the Auslander-Reiten quiver of $\rep(Q)$.

\begin{Theo} \label{Theorem2} Let $Q$ be connected infinite and strongly locally finite. The Auslander-Reiten quiver of $\rep(Q)$ has the following connected components.
\begin{enumerate}[$(1)$]
    \item A unique preprojective component $\mathcal{P}_Q$, which is a full predecessor closed subquiver of $\mathbb{N}Q^{\rm op}$. It is equal to $\mathbb{N}Q^{\rm op}$ if and only if $Q$ has no infinite sourced paths. All the finitely generated projective indecomposable representations lie in $\mathcal{P}_Q$.
\item A unique preinjective component $\mathcal{I}_Q$, which is a full successor closed subquiver of $\mathbb{N}^-Q^{\rm op}$. It is equal to $\mathbb{N}^-Q^{\rm op}$ if and only if $Q$ has no infinite sinked paths. All the finitely co-generated injective indecomposable representations lie in $\mathcal{I}_Q$.
\item An infinite number of quasi-wings, if $Q$ is not infinite Dynkin. Otherwise, no quasi-wing for type $\mathbb{A}_\infty$, one quasi-wing for type $\mathbb{D}_\infty$ and two quasi-wings for type $\mathbb{A}_\infty^\infty$. If there is a right infinite (resp. left infinite, finite) quasi-wing, then $Q$ has infinite sinked (resp. infinite sourced, both infinite sinked and infinite sourced) paths.
\item Some additional linear quivers, if and only if $Q$ is not a star quiver.
\end{enumerate}
\end{Theo}

\begin{Exam}
Let $Q$ be the quiver
$$\xymatrixcolsep{10pt}\xymatrixrowsep{2pt}
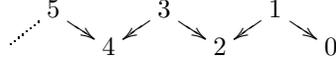
\xymatrix{& 5\ar@{.}[dl] \ar[dr] && 3\ar[dl] \ar[dr] && 1\ar[dl] \ar[dr]
& \\
&& 4 && 2 && 0}$$ For $i \ge 0$, let $M_i$ be
the indecomposable representation of $\rep(Q)$ with $M(j)=k$
for all $j \ge i$ and $M(j)=0$, otherwise. Since $Q$ has no infinite
path, $\rrep(Q)$ is the category of finite dimensional
representations. The only indecomposable infinite dimensional
representations of $\rep(Q)$, up to isomorphisms, are the $M_i$. The
only connected component of $\Gamma_{\rep(Q)}$ which does not contain finite dimensional representations is the following linear quiver:
$$\xymatrixrowsep{10 pt}\xymatrixcolsep{10 pt}\xymatrix{\cdots\ar[r] & M_4 \ar[r] &  M_2 \ar[r]  & M_0\ar[r]
 &  M_1 \ar[r] & M_3\ar[r] & M_5 \ar[r]   & \cdots \\
}
 $$
\end{Exam}

\begin{Exam}
Let $Q$ be the quiver
$$\xymatrixcolsep{10pt}\xymatrixrowsep{2pt}
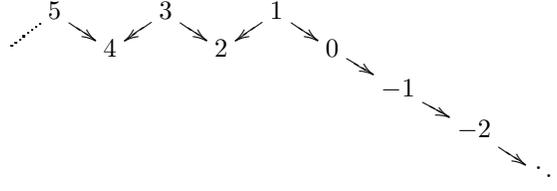
\xymatrix{& 5\ar@{.}[dl] \ar[dr] && 3\ar[dl] \ar[dr] && 1\ar[dl] \ar[dr]
& \\
&& 4 && 2 && 0\ar[dr] \\ &&  &&  && & -1\ar[dr] \\ &&  &&  && && -2\ar[dr] \\&&  &&  && && & \ddots}$$ For $i \in \Z$, let $M_i$ be
the indecomposable representation of $\rep(Q)$ with $M(j)=k$
for all $j \ge i$ and $M(j)=0$, otherwise. Denote by $M_\infty$ the indecomposable representation with $M(j)=k$ for all $j \in \Z$. Up to isomorphism, the only indecomposable
representations of $\rep(Q)$ which are not in $\rrep(Q)$ are the $M_i$ together with $M_\infty$. The connected components of $\Gamma_{\rep(Q)}$ which do not contain representations in $\rrep(Q)$ are the following linear quivers:
$$\xymatrixrowsep{10 pt}\xymatrixcolsep{10 pt}\xymatrix{\cdots\ar[r] & M_4 \ar[r] &  M_2 \ar[r]  & M_\infty}$$
$$\xymatrixrowsep{10 pt}\xymatrixcolsep{10 pt}\xymatrix{\cdots\ar[r]
 &  M_{-2} \ar[r] & M_{-1}\ar[r] & M_0 \ar[r]   & M_1 \ar[r] & M_3 \ar[r] & \cdots}
$$
\end{Exam}

\noindent{\emph{Acknowledgment}.} The author is supported by the Department of Mathematics at the University of Connecticut.

\end{document}